\documentclass[11pt,a4paper,oneside]{amsart}
\pdfoutput=1
\usepackage{amssymb}
\usepackage{esint}
\usepackage{mathtools}
\usepackage{hyperref}
\usepackage{enumitem}

\usepackage[english]{babel}
\usepackage[utf8]{inputenc} 
\usepackage[pdftex]{graphicx}

\theoremstyle{plain}
\newtheorem{theorem}{Theorem}[section]
\newtheorem{lemma}[theorem]{Lemma}
\newtheorem{corollary}[theorem]{Corollary}
\newtheorem{proposition}[theorem]{Proposition}

\theoremstyle{definition}

\theoremstyle{remark}
\newtheorem{remark}{Remark}[section]

\newcommand{\T}{\mathbb{T}}
\newcommand{\R}{\mathbb{R}}
\newcommand{\Z}{\mathbb{Z}}
\newcommand{\C}{\mathbb{C}}
\newcommand{\N}{\mathbb{N}}
\newcommand{\Q}{\mathbb{Q}}
\renewcommand{\P}{\mathbb{P}}
\newcommand{\Gr}{\mathbf{Gr}}

\newcommand{\vev}[1]{\left\langle#1\right\rangle}

\renewcommand{\phi}{\varphi}
\newcommand{\abs}[1]{\left| #1 \right|}
\newcommand{\norm}[1]{\Vert #1 \Vert}

\title[Fourier analysis of periodic Radon transforms]{Fourier analysis of periodic Radon transforms}

\author{Jesse Railo}
\address{Department of Mathematics and Statistics, University of Jyv\"askyl\"a, P.O.Box 35 (MaD) FI-40014 University of Jyv\"askyl\"a, Finland}
\email{jesse.t.railo@jyu.fi}
\date{\today}
\thanks{This is the accepted manuscript of the work that was first published in \emph{J. Fourier Anal. Appl.} 26 (2020), 64. \href{https://doi.org/10.1007/s00041-020-09775-1}{10.1007/s00041-020-09775-1}.}
\thanks{New contact details: Seminar for Applied Mathematics, Department of Mathematics, ETH Zurich, Rämistrasse 101, CH-8092 Zürich, Switzerland; \href{jesse.railo@math.ethz.ch}{jesse.railo@math.ethz.ch}.}

\begin{document}

\begin{abstract}
We study reconstruction of an unknown function from its $d$-plane Radon transform on the flat torus $\T^n = \R^n /\Z^n$ when $1 \leq d \leq n-1$. We prove new reconstruction formulas and stability results with respect to weighted Bessel potential norms. We solve the associated Tikhonov minimization problem on $H^s$ Sobolev spaces using the properties of the adjoint and normal operators. One of the inversion formulas implies that a compactly supported distribution on the plane with zero average is a weighted sum of its X-ray data.
\end{abstract}
\maketitle

\section{Introduction}

We study reconstruction of an unknown function from its $d$-plane Radon transform on the flat torus $\T^n = \R^n /\Z^n$ when $1\leq d \leq n-1$. The $d$-plane Radon transform of a function $f$ on $\T^n$ encodes the integrals of $f$ over all periodic $d$-planes. The usual $d$-plane Radon transform of compactly supported objects on $\R^n$ can be reduced into the periodic $d$-plane Radon transform, but not vice versa. This was demonstrated for the geodesic X-ray transform in the recent work of Ilmavirta, Koskela and Railo \cite{IKR19}. As general references on the Radon transforms, we point to \cite{H99,Q06,Hel13,KQ15}.

Reconstruction formulas for integrable functions and a family of regularization strategies considered in this article were derived in \cite{IKR19} for the geodesic X-ray transform ($d=1$) on $\T^2$. We extend these methods to the $d$-plane Radon transforms of higher dimensions, study new types of reconstruction formulas for distributions, and prove new stability estimates on the Bessel potential spaces. This article considers only the mathematical theory of Radon transforms on $\T^n$, whereas numerical algorithms (Torus CT) were implemented in \cite{IKR19,RK19}.

Injectivity, a reconstruction method and certain stability estimates of the $d$-plane Radon transform on $\T^n$ were proved for distributions by Ilmavirta in \cite{I15}. Our reconstruction formulas and stability estimates in this article are different than the ones in \cite{I15}. The first injectivity result for the geodesic X-ray transform on $\T^2$ was obtained by Strichartz in \cite{S82}, and generalized to $\T^n$ by Abouelaz and Rouvière in \cite{AR11} if the Fourier transform is $\ell^1(\Z^n)$. Abouelaz proved uniqueness under the same assumption for the $d$-plane Radon transform in \cite{A11}.

The X-ray transform and tensor tomography on $\T^n$ has been applied to other integral geometry problems. These examples include the broken ray transform on boxes \cite{I15}, the geodesic ray transform on Lie groups \cite{I16}, tensor tomography on periodic slabs \cite{IU18}, and the ray transforms on Minkowski tori \cite{I18}. We expect that the $d$-plane Radon transform on $\T^n$ has applications in similar and generalized geometric problems as well, but have not studied this possibility any further.

This article is organized as follows. The main results are stated in section \ref{sec:results}. We recall preliminaries and prove some basic properties in section \ref{sec:pre}. We prove new inversion formulas in section \ref{sec:L1}. We prove our stability estimates and theorems on Tikhonov regularization in section \ref{sec:stabtikh}.

\subsection{Results} \label{sec:results}

We describe our results next. Here we only briefly introduce the notation used, and more details are given in the subsequent sections. One can also find more details in \cite{I15,IKR19}. Let $n,d \in \Z$ be such that $n \geq 2$ and $1 \leq d \leq n-1$. We define the \textit{$d$-plane Radon transform} of $f \in \mathcal{T} := C^\infty(\T^n)$ as
\begin{equation}\label{eq:contRdef}R_df(x,A) := \int_{[0,1]^d} f(x+t_1v_1+\cdots+t_dv_d)dt_1\cdots dt_d\end{equation}
where $A = \{v_1,\dots,v_d\}$ is any set of linearly independent integer vectors $v_i \in \Z^n$. 

It can be shown that $A$ spans a periodic $d$-plane on $\T^n$, and on the other hand, any periodic $d$-plane on $\T^n$ has a basis of integer vectors. We can identify all periodic $d$-planes on $\T^n$ by the elements in the Grassmannian space $\Gr(d,n)$ which is the collection of all $d$-dimensional subspaces of $\Q^n$. We redefine the $d$-plane Radon transform on $\T^n$ as $R_df: \Gr(d,n) \to \C^\infty(\T^n)$ without a loss of data. The definition of $R_d$ extends to the periodic distributions $f \in \mathcal{T}'$ such that $R_df(\cdot,A) \in \mathcal{T}'$ for any $A \in \Gr(d,n)$. We use the shorter notations $R_{d,A}f = R_df(\cdot,A)$ and $X_{d,n} = \T^n \times \Gr(d,n)$.  More details are given in section \ref{sec:prtgr}.

Let $w: \Z^n \times \Gr(d,n) \to (0,\infty)$ be a weight function such that $w(\cdot,A)$ is at most of polynomial decay (\ref{eq:poldecdef}) for any fixed $A \in \Gr(d,n)$. If not said otherwise, then a weight $w$ is always assumed to be of this form. The associated Fourier multipliers on distributions are denoted by $F_w$. We denote the weighted Bessel potential space on the image side by $L_s^{p,l}(X_{d,n};w)$ where $s \in \R$, $p,l \in [1,\infty]$. The usual Bessel potential spaces on $\T^n$ are denoted by $L_s^p(\T^n)$, and $H^s(\T^n) = L_s^2(\T^n)$ is the fractional $L^2$ Sobolev space. The $L_s^{p,l}(X_{d,n};w)$ norms are $\ell^l$ norms over $\Gr(d,n)$ of the $w$-weighted Bessel potential norms of $L_s^p(\T^n;w(\cdot,A))$ with $A \in \Gr(d,n)$. More details are given in section \ref{sec:bessel}.

We show that $L_s^{p,l}(X_{d,n};w)$ are Banach spaces when $p \in [1,\infty]$ in lemma \ref{lem:banach}. Many of our results consider the Hilbert spaces with $p = l = 2$. Most of the theorems in this article would have been unreachable for $R_d$ when $d < n-1$ if we did not include weights in the data spaces. We construct weights which satisfy the assumptions of our theorems in section \ref{sec:constructions}.

\begin{remark} If $d =n-1$, then weights are not that important for the analysis of $R_d$ since $R_d$ maps $f \in H^s(\T^n)$ with $\hat{f}(0) = 0$ continuously to the natural image space $H^s(X_{d,n})$ without setting any weight. Therefore weights are only required at the origin on the Fourier side of the data space. This was demonstrated in the case of $n =2$ and $d=1$ in \cite{IKR19}, or for example in the special case \eqref{eq:recwithoutFourier0mean} of theorem \ref{thm:recwithoutFourier}. \end{remark}

Our first theorem considers the adjoint and the normal operators of $R_d: H^s(\T^n) \to L_s^{2,2}(X_{d,n};w)$. This generalizes \cite[Proposition 11]{IKR19} into higher dimensions. Theorem \ref{thm:adjoint} and corollary \ref{cor:h2results} are proved in section \ref{sec:adjnorm}.

\begin{theorem}[Adjoint and normal operators]\label{thm:adjoint} Let $s \in \R$ and suppose that there exists $C_w > 0$ such that \begin{equation}\sum_{A \in \Omega_k} w(k,A)^2 \leq C_w^2, \quad \Omega_k := \{\, A \in \Gr(d,n) \,;\, k\bot A\,\}\end{equation}
for any $k \in \Z^n$. Then the adjoint of $R_{d}: H^s(\T^n) \to L_s^{2,2}(X_{d,n};w)$ is given by
\begin{equation}\widehat{R_d^*g}(k) = \sum_{A \in \Omega_k} w(k,A)^2\hat{g}(k,A)\end{equation}
and the normal operator $R_d^*R_d: H^s(\T^n) \to H^s(\T^n)$ is the Fourier multiplier associated with $W_k := \sum_{A \in \Omega_k} w(k,A)^2$. In particular, the mapping  $F_{W_k^{-1}}R_d^*: R_d(\mathcal{T}') \to \mathcal{T}'$ is the inverse of $R_d$.
\end{theorem}

Theorem \ref{thm:adjoint} gives a new inversion formula in terms of the adjoint and a Fourier multiplier. Its corollary \ref{cor:h2results} gives new stability estimates on $H^s(\T^n)$. The stability estimates of $R_1$ on $H^s(\T^2)$ were not explicitly written down in \cite{IKR19} but they can be found between the lines. We denote by $R_d^{*,w}$ the adjoint of $R_d$ associated to the weight $w$ when the weight needs to be specified.

\begin{corollary}[Stability estimates]\label{cor:h2results}Suppose that the assumptions of theorem \ref{thm:adjoint} hold, and that there exists $c_w > 0$ such that $W_k \geq c_w^2$ for any $k \in \Z^n$. \begin{enumerate}[label=(\roman*)]
\item  Then $F_{W_k^{-1}}R_d^*: L_s^{2,2}(X_{d,n};w) \to H^s(\T^n)$ is $1/c_w$-Lipschitz.\label{item:h1prop1}
\item Let $f \in \mathcal{T}'$. Then
\begin{equation}\label{eq:normalopcor}\norm{f}_{H^s(\T^n)} \leq \frac{1}{c_w}\norm{R_df}_{L_s^{2,2}(X_{d,n};w)}.\end{equation}\label{item:h1prop2}
\item Let $\tilde{w}(k,A) = \frac{w(k,A)}{\sqrt{W_k}}$ and $p \in [1,\infty]$. Then $R_d^{*,\tilde{w}}R_df = f$ and $\norm{f}_{L_s^p(\T^n)} = \norm{R_d^{*,\tilde{w}}R_df}_{L_s^p(\T^n)}$ for any $f \in \mathcal{T}'$.\label{item:Lspstab}
\end{enumerate}
\end{corollary}

In order to prove $L_s^p \lesssim L_s^p$ type stability \ref{item:Lspstab} for more general weights in terms of the normal operator, one would have to show that $F_{W_k^{-1}}$ is a bounded $L^p$ multiplier. Other stability estimates on $L_s^p(\T^n)$ are given in terms of $R_df$ in proposition \ref{prop:stabilitymain}. These stability estimates follow from corollary \ref{cor:h2results} and the Sobolev inequality on $\T^n$. This method requires additional smoothness of $R_df$ in order to control the norm of $f$ due to the use of the Sobolev inequality.

We have proved three other new inversion formulas for $R_d$ as well. The other two inversion formulas are given in proposition \ref{prop:L1rec} and its corollary \ref{cor:dist}. Proposition \ref{prop:L1rec} generalizes the inversion formula \cite[Theorem 1]{IKR19} into higher dimensions. Its corollary \ref{cor:dist} generalizes the formula for all periodic distributions using the structure theorem. We state the third inversion formula here since we find it to be the most interesting one. Theorem \ref{thm:recwithoutFourier} is proved in the end of section \ref{sec:L1}.

\begin{theorem}[Periodic filtered backprojections]\label{thm:recwithoutFourier} Suppose that $f \in \mathcal{T}'$. Let $w: \Z^n \times \Gr(d,n) \to \R$ be a weight so that
\begin{equation}\sum_{A \in \Omega_k} w(k,A) = 1, \quad \Omega_k := \{\, A \in \Gr(d,n) \,;\, k\bot A\,\}\end{equation} and the series is absolutely convergent for any $k \in \Z^n$. (The weight does not have to generate a norm or have at most of polynomial decay.) Then
\begin{equation}(f,h) = \sum_{A \in \Gr(d,n)}(F_{w(\cdot,A)}R_{d,A}f,h), \quad \forall h \in C^\infty(\T^n).\end{equation}
Moreover, if $f$ has zero average and $d = n-1$, then \begin{equation}\label{eq:recwithoutFourier0mean}f = \sum_{A \in \Gr(d,n)} R_{d,A}f.\end{equation}
\end{theorem}

\begin{remark} The author is not aware of a similar formula for the inverse Radon transform in earlier literature. We emphasize that this new result implies that a clever sum of the $(n-1)$-plane Radon transform data is the target function. If $n=2$, this holds true for the X-ray transform of compactly supported functions on the plane $\R^2$. We further remark that it is easy to recover the average of a function and filter it out from $R_{n-1}f$.
\end{remark}

Finally, we state our results on regularization. These results generalize \cite[Theorems 2 and 3]{IKR19} into higher dimensions. The proofs are given in section \ref{sec:stabtikh}.  Let $g \in L_r^{2,l}(X_{d,n};w)$. We consider the \textit{Tikhonov minimization problem}
\begin{equation}\label{eq:tikh}
\underset{f \in H^t(\T^n)}{\arg\min} \left(\norm{R_df-g}_{L_r^{2,l}(X_{d,n};w)}^l + \alpha \norm{f}_{H^s(\T^n)}^2\right).
\end{equation}
for any $n \geq 2$, $1 \leq d \leq n-1$, $\alpha > 0$, $l = 2$, and $r, s, t\in \R$. We do not fix the regularity of $f$ a priori but the space $H^t(\T^n)$ will be found after solving the minimization problem for distributions in general.

Let $w$ be a weight, $z \in \R$, and $\alpha >0$. We define the operator $P_{w,z}^\alpha: \mathcal{T}' \to \mathcal{T}'$ to be the Fourier multiplier associated with
\begin{equation}p_{w,z}^\alpha(k) := \frac{1}{W_k+\alpha\vev{k}^{2z}}.
\end{equation}

\begin{theorem}[Tikhonov minimization problem]\label{thm:tikhmin} Let $w$ be a weight such that $c_w^2 \leq W_k \leq C_w^2$ for some uniform constants $c_w, C_w > 0$. Suppose that $\alpha >0$, and $s \geq r$. Then the unique minimizer of the Tikhonov minimization problem (\ref{eq:tikh}) with $g \in L_r^{2,2}(X_{d,n};w)$ is given by $f = P_{w,s-r}^\alpha R_d^*g \in H^{2s-r}(\T^n)$.
\end{theorem}

The last theorem we state in the introduction generalizes the result \cite[Theorem 3]{IKR19} on regularization strategies to higher dimensions.

\begin{theorem}[Regularization strategy]\label{thm:regstrat} Let $w$ be a weight such that $c_w^2 \leq W_k \leq C_w^2$ for some uniform constants $c_w, C_w > 0$. Suppose $r,t,s,\delta \in \R$ are constants such that $2s+t \geq r$, $\delta \geq 0$, and $s > 0$. Let $g \in L_t^{2,2}(X_{d,n}; w)$ and $f \in H^{r+\delta}(\T^n)$.

Then the Tikhonov regularized reconstruction operator $P_{w,s}^\alpha R_d^*$ is a regularization strategy in the sense that
\begin{equation}\lim_{\epsilon \to 0} \sup_{\norm{g}_{L_t^{2,2}(X_{d,n}; w)}\leq \epsilon} \norm{P_{w,s}^{\alpha(\epsilon)} R_d^*(R_df+g)-f}_{H^r(\T^n)} = 0\end{equation} where $\alpha(\epsilon) =\sqrt{\epsilon}$ is an admissible choice of the regularization parameter.

Moreover, if $\norm{g}_{L_t^{2,2}(X_{d,n};w)} \leq \epsilon$, $0 < \delta < 2s$, and $0 < \alpha \leq c_w^2(2s/\delta-1)$, we have a quantitative convergence rate
\begin{equation}\label{eq:quantitative}\begin{split}
&\norm{P_{w,s}^\alpha R_d^*(R_d f +g) -f}_{H^r(\T^n)} \\
&\,\,\leq \alpha^{\delta/2s}c_w^{-\delta/s}C(\delta/2s)\norm{f}_{H^{r+\delta}(\T^n)} + C_w^3c_w^{-2}\frac{\epsilon}{\alpha}\end{split}\end{equation}
where $C(x) = x(x^{-1}-1)^{1-x}$.
\end{theorem}

\begin{remark} The optimal rate of convergence with respect to $\epsilon >0$ can be found by choosing the regularization parameter $\alpha(\epsilon)$ so that the terms on the right hand side of (\ref{eq:quantitative}) are of the same order.
\end{remark}

\begin{flushleft}
\textbf{Acknowledgements.} This work was supported by the Academy of Finland (Center of Excellence in Inverse Modelling and Imaging, grant numbers 284715 and 309963). The author is grateful to Joonas Ilmavirta who has shared his insight of the questions studied in the article. The author wishes to thank Giovanni Covi, Keijo Mönkkönen and Mikko Salo for their valuable comments on the manuscript and suggestions for improvements. The author thanks the anonymous referees for their helpful comments.
\end{flushleft}

\section{Preliminaries}\label{sec:pre}
\subsection{Periodic Radon transforms and Grassmannians}\label{sec:prtgr}

We denote by $\mathcal{T}$ the set $C^\infty(\T^n)$ and $\mathcal{T}'$ its dual space, i.e.~the space of periodic distributions. Denote by $G_d^n$ the set of linearly independent unordered $d$-tuples in $\Z^n\setminus 0$. We may write any element $A \in G_d^n$ as $A = \{v_1,\dots,v_d\}$. The elements in the set $G_d^n$ span all periodic $d$-planes on $\T^n$.

Suppose that $f \in \mathcal{T}$. We define the \textit{$d$-plane Radon transform} of $f$ as
\begin{equation}R_df(x,A) := \int_{[0,1]^d} f(x+t_1v_1+\cdots+t_dv_d)dt_1\cdots dt_d.\end{equation} We remark that $R_d: \mathcal{T} \to \mathcal{T}^{G_d^n}$, $R_df: \T^n \times G_d^n \to \C$ and $R_df(\cdot,A): \T^n \to \C$.

Denote the duality pairing between $\mathcal{T}'$ and $\mathcal{T}$ by $(\cdot,\cdot)$. If $f,g \in \mathcal{T}$, it follows easily from Fubini's theorem that \begin{equation}(f,R_dg(\cdot,A)) = (R_df(\cdot,A),g).\end{equation} We define the $d$-plane Radon transform for any $f \in \mathcal{T}'$ and $A \in G_d^n$ simply as
\begin{equation} (R_df(\cdot,A))(g) = (f,R_dg(\cdot,A)) \quad \forall g \in \mathcal{T}.
\end{equation} This is the unique continuous extension of $R_d(\cdot,A)$ to the periodic distributions. The Fourier series coefficients of $R_df(\cdot,A)$ are defined as usual.

We denote the \textit{Grassmannian} of $d$-dimensional subspaces of $\Q^n$ by $\Gr(d,n)$. If $A, B \in G_d^n$ span the same subspace of $\Q^n$, then $A$ and $B$ represent the same element in $\Gr(d,n)$, and $R_df(\cdot,A) = R_df(\cdot,B)$ holds for any $f \in \mathcal{T}'$ by theorem \ref{thm:ilmavirta}. On the other hand, for every $A \in \Gr(d,n)$ there exists $\tilde{A} \in G_d^n$ that spans $A$. This allows one to define the Radon transform as $R_df: \Gr(d,n) \to \mathcal{T}'$ without data redundancy by setting $R_df(\cdot,A) := R_df(\cdot,\tilde{A})$ where $\tilde{A} \in G_d^n$ spans $A \in \Gr(d,n)$. This connection to the Grassmannians was mentioned earlier in \cite{I15} but was not directly used.

\begin{remark} Let us denote the \textit{projective space} $\P^{n-1} := \Gr(1,n)$. The \textit{height} of $P \in \P^{n-1}$ is defined by $H(P) = \gcd(p)^{-1}\abs{p}_{\ell^\infty}$ using any representative $p$ of $P$. The projective space $\P^1$ and the height were used in \cite{IKR19} to analyze the number of projection directions required to reconstruct the Fourier series coefficients of a phantom up to a fixed radius. This question reduces to Schanuel's theorem \cite{S64} in algebraic number theory. This analysis in \cite{IKR19} extends to higher dimensions when $d = n-1$.
\end{remark}

\subsection{Bessel potential spaces and data spaces}
\label{sec:bessel}

Let $f \in \mathcal{T}'$. We mean by the expression $\sum_{k \in \Z^n} \vev{k}^s \hat{f}(k)e^{2\pi i k \cdot x}$ the limit \begin{equation}\tilde{f}(x) := \lim_{r \to \infty} f_{r,s}(x),\quad f_{r,s}(x):= \sum_{\abs{k}_{\ell^\infty(\Z^n)} \leq r} \vev{k}^s \hat{f}(k)e^{2\pi i k \cdot x},\end{equation}
in the sense of distributions. If $f \in L^p(\T^n)$ with $p \in (1,\infty)$, then $f_{r,0} \to f$ in $L^p(\T^n)$ as $r\to \infty$. Moreover, if $p \in (1,\infty]$, then $\tilde{f} = f$ almost everywhere as the pointwise limit by a higher dimensional Carleson-Hunt theorem. These facts are proved for example in \cite[Theorems 4.2 and 4.3]{F12}. If $p=1$, one can utilize the Cesàro sums to reconstruct a distribution in $L^1(\T^n)$ from its Fourier series.

For any Sobolev scale $s \in \R$, we define the \textit{Bessel potential spaces} $L_s^p(\T^n) \subset \mathcal{T}'$ by the relation $f \in L_s^p(\T^n)$ if and only if $(1-\Delta)^{s/2}f \in L^p(\T^n)$ (see e.g. \cite{BT13}). We define the \textit{Bessel potential norms} by \begin{equation}\norm{f}_{L_s^p(\T^n)} := \norm{(1-\Delta)^{s/2}f}_{L^p(\T^n)}.\end{equation} Then the space $L_s^p(\T^n) \subset \mathcal{T}'$ consists of all $f \in \mathcal{T}'$ with $\norm{f}_{L_s^p(\T^n)} < \infty$. If $p \in (1,\infty)$ and $s \in \R$, then
\begin{equation}\begin{split}\norm{f}_{L_s^p(\T^n)} &= \lim_{r\to\infty}\norm{\sum_{\abs{k}_{\ell^\infty(\Z^n)}\leq r} \vev{k}^s \hat{f}(k)e^{2\pi i k \cdot x}}_{L^p(\T^n)}, \\
\norm{f}_{H^s(\T^n)} &= \sqrt{\sum_{k \in \Z^k} \vev{k}^{2s}\abs{\hat{f}(k)}^2}\end{split}\end{equation} where $\vev{k}=(1+\abs{k}^2)^{1/2}$ as usual. When $p \in (1,\infty)$, one has equivalently that $f \in L_s^p(\T^n)$ if and only if $(1-\Delta)^{s/2}f \in L^p(\T^n)$ in terms of the $L^p$ convergent Fourier series and $f\in \mathcal{T}'$. Moreover, for any $p \in (1,\infty]$ and $f \in L_s^p(\T^n)$ it holds that \begin{equation}\norm{f}_{L_s^p(\T^n)} = \norm{\lim_{r\to\infty}\sum_{\abs{k}_{\ell^\infty(\Z^n)}\leq r} \vev{k}^s \hat{f}(k)e^{2\pi i k \cdot x}}_{L^p(\T^n)}\end{equation} where the limit is taken pointwise since the Fourier series converges almost everywhere. If $p = 2$, then $H^s(\T^n) = L_s^p(\T^n)$ is the fractional $L^2$ Sobolev space. If $p \in [1,\infty]$ and $s = 0$, then the $L_0^p(\T^n)$ and $L^p(\T^n)$ norms agree. The Bessel potential spaces are used as domains of $R_d$ in this work, which extends studies of the case $p = 2$ in \cite{I15, IKR19}.

If $\omega: \Z^n \to (0,\infty)$ and $f \in \mathcal{T}'$, then we define the $\omega$-weighted norms by
\begin{equation}\norm{f}_{L_s^p(\T^n; \omega)} := \norm{F_\omega f}_{L_s^p(\T^n)}\end{equation}
where $F_\omega$ is the Fourier multiplier of $\omega$. We say that a weight $\omega: \Z^n \to (0,\infty)$ is \textit{at most of polynomial decay} if there exists $C, N > 0$ such that
\begin{equation}\label{eq:poldecdef}\omega(k) \geq C\vev{k}^{-N} \quad \forall k \in \Z^n.\end{equation}

We next define suitable data spaces that contain ranges of $R_d$ when its domains are restricted to the Bessel potential spaces. Let us denote $X_{d,n} := \T^n \times \Gr(d,n)$ to keep our notation shorter. We generalize the data space given in \cite{IKR19} to all $n \geq 2$, $1 \leq d\leq n-1$, and $p \in [1,\infty]$, using the Grassmannians, the Bessel potential spaces and weights.

Let $1 \leq d \leq n-1$ and $w: \Z^n \times \Gr(d,n) \to (0,\infty)$ be a weight function such that $w(\cdot,A)$ is at most of polynomial decay for any fixed $A \in \Gr(d,n)$. We always assume in this work that the weight is at most of polynomial decay. We say that a (generalized) function $g: X_{d,n} \to \C$ belongs to $L_{s}^{p,l}(X_{d,n}; w)$ with $1 \leq l < \infty$ if the norm
\begin{equation} \norm{g}_{L_{s}^{p,l}(X_{d,n}; w)}^l := \sum_{A \in \Gr(d,n)} \norm{g(\cdot,A)}_{L_s^p(\T^n; w(\cdot,A))}^l
\end{equation}
is finite and $g(\cdot,A) \in \mathcal{T}'$ for any fixed $A \in \Gr(d,n)$. Similarly, if $l = \infty$, we define
\begin{equation} \norm{g}_{L_{s}^{p,\infty}(X_{d,n}; w)} := \sup_{A \in \Gr(d,n)} \norm{g(\cdot,A)}_{L_s^p(\T^n; w(\cdot,A))}
\end{equation}
In the above definition, one can replace $\Gr(d,n)$ by any countable set $Y$ (cf. lemma \ref{lem:banach}).

If $p,l = 2$, then the norm is generated by the inner product
\begin{equation}\label{eq:innerprod}(h,g)_{L_s^{2,2}(X_{d,n};w)} := \sum_{A \in \Gr(d,n)} (F_{w(\cdot,A)}h,F_{w(\cdot,A)}g)_{H^s(\T^n)}\end{equation} which makes $L_s^{2,2}(X_{d,n};w)$ a Hilbert space. We prove that the spaces $L_s^{p,l}(X_{d,n};w)$ are Banach spaces when $p \in [1,\infty]$ in lemma \ref{lem:banach}. We emphasize that a weight does not have to have uniform coefficients for its at most of polynomial decay with respect to $\Gr(d,n)$.

There is a connection to the norms used in \cite{IKR19}. Let $w$ be any weight such that $\sum_{A \in \Gr(1,2)} w(0,A)^2 = 1$, and $w(k,A) \equiv 1$ if $k\neq 0$. Now the results in \cite{IKR19} follow from the results of this article using the norm $L_{s}^{2,2}(X_{1,2}; w)$ as the image side spaces in \cite{IKR19} are contained in $L_{s}^{2,2}(X_{1,2}; w)$. 

Yet another norm was used for the stability estimates in \cite{I15}. In the cases $d = n-1$ and $l = \infty$, our analysis of $R_d$ would not require weights, and can be performed similarly to \cite{I15,IKR19}. The analysis of $R_d|_{L_s^p(\T^n)}$ has not been done before if $p \neq 2$. The Bessel potential norms on the domain side are used to understand better the mapping properties of $R_{d}$.

We state and prove the following lemma for the sake of completeness. We remark that without the decay condition on weights these weighted spaces would not be complete.

\begin{lemma}\label{lem:banach} Let $Y$ be a countable set. Let $w: \Z^n \times Y \to (0,\infty)$ be a weight that is at most of polynomial decay for any fixed $y \in Y$. Suppose that $s \in \R, p \in [1,\infty], l \in [1,\infty]$, and $n \geq 1$. Then $L_{s}^{p,l}(\T^n \times Y; w)$ is a Banach space. In particular, $L_{s}^{2,2}(\T^n \times Y; w)$ is a Hilbert space.
\end{lemma}
\begin{proof} Suppose that $1\leq l < \infty$. (If $l = \infty$, the proof is similar.) We first show that $L_{s}^{p,l}(\T^n \times Y; w)$ is a vector space. Let $c \in \C$ and $f,g \in L_{s}^{p,l}(\T^n \times Y; w)$. We have trivially that
\begin{equation}\norm{cf}_{L_{s}^{p,l}(\T^n \times Y; w)}^l = \abs{c}^l\sum_{y \in Y} \norm{f(\cdot,y)}_{L_s^p(\T^n;w)}^l.\end{equation} The Minkowski and triangle inequalities imply
\begin{equation}\begin{split}\norm{f+g}_{L_{s}^{p,l}(\T^n \times Y; w)} &= \left(\sum_{y \in Y} \norm{F_{w(\cdot,y)}f(\cdot,y)+F_{w(\cdot,y)}g(\cdot,y)}_{L_s^p(\T^n)}^l\right)^{1/l} \\
&\leq \norm{f}_{L_{s}^{p,l}(\T^n \times Y; w)} + \norm{g}_{L_{s}^{p,l}(\T^n \times Y; w)}.
\end{split}\end{equation}
This shows that $L_{s}^{p,l}(\T^n \times Y; w)$ is a vector subspace of all collections of distributions $\{f(\cdot,y)\}_{y \in Y}$ with $f(\cdot,y) \in \mathcal{T}'$.

We show next that $L_{s}^{p,l}(\T^n \times Y; w)$ is a complete space. Let $f_i \in L_{s}^{p,l}(\T^n \times Y; w)$ be a Cauchy sequence. It follows from the definition of the norm in $L_{s}^{p,l}(\T^n \times Y; w)$ that $f_i(\cdot,y) \in L_s^p(\T^n; w(\cdot,y))$ is a Cauchy sequence for any $y \in Y$. Suppose that each $L_s^p(\T^n; w(\cdot,y))$ is complete. It follows that $f_i(\cdot,y) \to f_y \in L_s^p(\T^n; w(\cdot,y))$ as $i \to \infty$. This implies that there exists a limit of $f_i$ in $L_{s}^{p,l}(\T^n \times Y; w)$ by standard arguments.

Let us prove that $L_s^p(\T^n; w(\cdot,y))$ is complete for any $y \in Y$. Take a Cauchy sequence $f_i \in L_s^p(\T^n; w(\cdot,y))$. Now it follows that the distributions
\begin{equation}g_i = (1-\Delta)^{s/2}F_{w(\cdot,y)}f\end{equation} are in $L^p(\T^n)$ and form a Cauchy sequence. Therefore $\lim_{i \to \infty} g_i =: g$ exists. We claim that the distribution defined on the Fourier side as $\hat{f}(k) := \frac{\hat{g}(k)}{\vev{k}^s w(k,y)}$ is the limit of $f_i$ in $L_s^p(\T^n; w(\cdot,y))$. 

We need to show two things, that $f \in \mathcal{T}'$ and $\norm{f_i-f}_{L_s^p(\T^n; w(\cdot,y))} \to 0$ as $i \to \infty$. We first notice that $(1-\Delta)^{s/2}F_{w(\cdot,y)}f = g$ belongs to $L^p(\T^n)$. We can now calculate that \begin{equation}\begin{split}
\norm{f_i-f}_{L_s^p(\T^n; w(\cdot,y))} &= \norm{(1-\Delta)^{s/2}F_{w(\cdot,y)}(f_i-f)}_{L^p(\T^n)} \\
&=\norm{g_i -g}_{L^p(\T^n)}
\end{split}
\end{equation}
for any $i \in \N$. Therefore, $\norm{f_i-f}_{L_s^p(\T^n; w(\cdot,x))} \to 0$ as $i \to \infty$. 

It is enough that the Fourier coefficients of $f$ have polynomial growth by the structure theorem of periodic distributions \cite[Chapter 3.2.3]{ST87}. We have $\abs{\hat{g}(k)} \leq C_1 \vev{k}^\alpha$ for some $\alpha, C_1 > 0$ since $g \in L^p(\T^n) \subset \mathcal{T}'$. On the other hand, we assumed that $w(k,y) \geq C_2\vev{k}^{-N}$ for some $C_2, N > 0$. Hence, we obtain that
\begin{equation}\abs{\hat{f}(k)} = \abs{\frac{\hat{g}(k)}{\vev{k}^s w(k,y)}} \leq (C_1/C_2)\vev{k}^{\alpha+N-s}.
\end{equation}
This shows that $f \in \mathcal{T}'$.
\end{proof}

\begin{remark} One uses the fact that weights have at most of polynomial decay only to show that the limits of Cauchy sequences are in $\mathcal{T}'$. One could also allow more rapid decay for weights but in that case, the objects of the completion would not be distributions but ultra-distributions \cite{ST87}. In the analysis of $R_d$, such generality seems to be unnecessary and our assumptions avoid this.
\end{remark}

\subsection{On constructions of weights}\label{sec:constructions}

In this section, we discuss how to construct weights that satisfy the assumptions of our theorems. The weights of this paper are of the form $w: \Z^n \times \Gr(d,n) \to (0,\infty)$ with the following properties.
\begin{enumerate}[label=(\roman*)]
\item For any $A \in \Gr(d,n)$ there exists $C, N > 0$ such that $w(k,A) \geq C\vev{k}^{-N}$ for every $k \in \Z^n$. \label{item:poldec}
\item There exists $C > 0$ such that $W_k \leq C$ for every $k \in \Z^n$ where $W_k = \sum_{A \in \Omega_k} w(k,A)^2$ and $\Omega_k = \{\, A \in \Gr(d,n) \,;\, k \bot A\,\}$.\label{item:upbdd}
\item There exists $c > 0$ such that $c \leq W_k$ for every $k \in \Z^n$. \label{item:lvbdd}
\end{enumerate}

The property \ref{item:poldec} is assumed for any weight in this article to guarantee that $L_s^{p,l}(X_{d,n};w)$ are Banach spaces. The property \ref{item:upbdd} is assumed for most of the weights to guarantee that $R_d: L_s^p(\T^n) \to L_s^{p,l}(X_{d,n};w)$ is continuous (with some restrictions if $p,l \neq 2$). The property \ref{item:lvbdd} is additionally assumed to prove the stability estimates and the theorems on regularization.

First of all, it is very easy to construct weights that satisfy \ref{item:poldec} alone. It is not hard to construct weights that satisfy \ref{item:poldec} and \ref{item:upbdd}. Since the set $\Gr(d,n)$ is countable, we may write it with an enumeration $\phi: \Gr(d,n) \to \N$. For example, we construct a weight $w(k,A) =  2^{-\phi(A)}\vev{k}^{-N}$ with large enough $N > 0$ chosen such that $\sum_{k \in \Z^n} \vev{k}^{-2N} < \infty$. Then $\sum_{A\in \Gr(d,n)}\sum_{k \in \Z^n} w(k,A)^2 < C$ for some $C > 0$. This shows that both conditions \ref{item:poldec} and \ref{item:upbdd} hold.

We give next a nontrivial example of a weight satisfying \ref{item:upbdd} and \ref{item:lvbdd} but not (necessarily) \ref{item:poldec}. Let $\phi_k: \Omega_k \to \N$ be an enumeration. Let $Q := \{\,(k,A) \in \Z^n \times \Gr(d,n)\,;\, A \in \Omega_k\,\}$. For any $(k,A) \in Q$, we define the weight $w(k,A) := \frac{h(k)}{\phi_k(A)^{1/2 +\epsilon}}$ with some mapping $h: \Z^n \to (a,b)$ with $0<a\leq b <\infty$ and $\epsilon > 0$. If $(k,A) \notin Q$, we set $w(k,A) = 1$. One has that $\abs{\Omega_k} = \infty$ if $1\leq d < n-1$ or $k = 0$, and $\abs{\Omega_k} = 1$ if $d = n-1$ and $k \neq 0$. Now \begin{equation}\sum_{A\in \Omega_k} w(k,A)^2 = h^2(k)\sum_{i=1}^{\abs{\Omega_k}} i^{-1-2\epsilon}.\end{equation} Hence, we get that $a^2 \leq W_k \leq Cb^2$ where $C= \sum_{i=1}^\infty i^{-1-2\epsilon}$.

The problem gets more difficult if the all three conditions must be satisfied at the same time. We solve this problem now by combining ideas from the both constructions above. We make a proposition about a concrete example, and more general methods are summarized in remarks \ref{rm:joku} and \ref{rm:joku2}.

\begin{proposition}\label{prop:constructions} Let $\phi_k: \Omega_k \to \N$ be an enumeration for any $k \in \Z^n$, and let $\phi: \Gr(d,n) \to \N$ be an enumeration. Let $h: \Z^k \to (a,b)$ with $0<a\leq b < \infty$ and $g(k) = \vev{k}^{-N}$ for some $N \geq 0$. Then the weight
\begin{equation} \label{eq:goodweight}w(k,A) := \begin{cases} \frac{h(k)}{\phi_k(A)} + \frac{g(k)}{\phi(A)} \quad & (k,A) \in Q \\
1 \quad & (k,A) \in Q^c\end{cases}
\end{equation}
satisfies the properties \ref{item:poldec}, \ref{item:upbdd} and \ref{item:lvbdd}.
\end{proposition}
\begin{proof} Using the definition (\ref{eq:goodweight}) and the positivity of the involved functions, we have that \begin{equation}
W_k \geq h^2(k)\sum_{A \in \Omega_k} \phi_k(A)^{-2} = h^2(k)\sum_{i=1}^{\abs{\Omega_k}} i^{-2} \geq a^2.\end{equation} This shows \ref{item:lvbdd}.

Suppose that $(k,A) \in Q$. We use \begin{equation}\label{eq:propA}\frac{1}{2}w(k,A)^2 \leq \frac{h^2(k)}{\phi_k(A)^{2}}+\frac{g^2(k)}{\phi(A)^{2}}\end{equation} to estimate $W_k$ from above. The formula (\ref{eq:propA}) gives
\begin{equation} \frac{1}{2}W_k \leq \sum_{A \in \Omega_k}\left(\frac{h^2(k)}{\phi_k(A)^{2}}+\frac{g^2(k)}{\phi(A)^{2}}\right) \leq h^2(k)\sum_{i=1}^{\abs{\Omega_k}} i^{-2} + \vev{k}^{-2N}\sum_{i=1}^{\abs{\Omega_k}} i^{-2}.
\end{equation} Since $\vev{k}^{-2N} \leq 1$ and $h(k) \leq b$ for any $k \in \Z^n$, we obtain that $W_k \leq 2C(1+b^2)$ where $C = \sum_{i=1}^{\infty} i^{-2} < \infty$. This shows \ref{item:upbdd}.

Using the definition (\ref{eq:goodweight}) and the positivity of the involved functions, we can directly estimate that 
\begin{equation}\abs{w(k,A)} \geq \min\{1,\frac{1}{\phi(A)}\vev{k}^{-N}\} =\frac{1}{\phi(A)}\vev{k}^{-N}.\end{equation}
This shows that $w(\cdot,A)$ is at most of polynomial decay \ref{item:poldec}.
\end{proof}

\begin{remark}\label{rm:joku} Proposition \ref{prop:constructions} generalizes for $w(k,A)|_Q = h(k)\psi(k,A) + g(k)\omega(A)$ with the conditions that $h(k)$ is bounded from above and below, $g(k)$ has at most of polynomial decay and is bounded above, the sums of $\omega(A)^2$ over $\Omega_k$ are uniformly bounded from above, and the sums of $\psi(k,A)^2$ over $\Omega_k$ are uniformly bounded from below and above.
\end{remark}

\begin{remark}\label{rm:joku2} If a weight $w$ satisfies the conditions \ref{item:poldec} and \ref{item:upbdd}, then it can be normalized as $\tilde{w}(k,A) := \frac{w(k,A)}{\sqrt{W_k}}$. The normalized weight $\tilde{w}$ has the property that $\tilde{W}_k = 1$ for any $k \in \Z^n$. Moreover, since $w(k,A)$ is at most of polynomial decay and $\sqrt{W_k} \leq C$ for some $C>0$, it follows that $\tilde{w}$ is at most of polynomial decay.
\end{remark}

We can construct weights that satisfty the assumptions of theorem \ref{thm:recwithoutFourier} by defining $w(k,A) = 2^{-\phi_k(A)}$ for any $(k,A) \in Q$ and $w(k,A) = 1$ if $(k,A) \notin Q$. If $d < n-1$, then $\sum_{A \in \Omega_k} w(k,A) = 1$ for any $k \in \Z^n$, and the series $\sum_{A \in \Omega_k} w(k,A)$ are absolutely convergent.

\subsection{Basic properties of periodic Radon transforms}

In this section, we state and prove some basic properties of $R_d$. Some of these properties were used earlier in the special cases in \cite{I15,IKR19}. We have chosen to include most of the proofs here for completeness.

\subsubsection{Periodic Radon transforms for integrable functions}

Let $T = (t_1,\dots,t_d) \in \R^d$ and $A = \{v_1,\dots,v_d\} \in G_d^n$. We can define $R_df(\cdot,A)$ for $L^1(\T^n)$ functions simply as
\begin{equation}\label{eq:aeL1def}R_{d,A}f(x) := \int_{[0,1]^d} f(x+t_1v_1+\cdots+t_dv_d)dt_1\cdots dt_d\end{equation} where the formula is defined for a.e. $x \in \T^n$. We lighten our notation by denoting the corresponding linear combinations by $T\cdot A = t_1v_1+\cdots+t_dv_d$ with respect to the enumeration of $A$. The following basic properties are valid.

\begin{lemma}\label{lem:L1lemma} Suppose that $f \in L^1(\T^n)$ and $A \in G_d^n$. Then $R_{d,A}f$ can be defined by the formula (\ref{eq:aeL1def}) for a.e. $x \in \T^n$. Moreover,
\begin{enumerate}[label=(\roman*)]
\item this definition coincides with the distributional definition: for every $f \in L^1(\T^n)$ and $g \in L^\infty(\T^n)$ it holds that $(R_{d,A} f,g) = (f,R_{d,A} g)$; \label{item:L1prop1}
\item $R_{d,A}: L^p(\T^n) \to L^p(\T^n)$ is $1$-Lipschitz for any $p \in [1,\infty]$.\label{item:L1prop2}
\item Suppose that $f \in \mathcal{T}'$, $A \in G_d^n$ and $R_df(\cdot,A) \in L^1(\T^n)$. Then $R_{d,A}f(x+S\cdot A) = R_{d,A}f(x)$ for a.e. $x\in \T^n$ and every $S \in \R^d$.\label{item:L1prop3}
\end{enumerate}
\end{lemma}

We postpone the proof of lemma \ref{lem:L1lemma} for a while. We remark that lemma \ref{lem:L1lemma} is a simple generalization of \cite[Lemma 7]{IKR19}, which was stated in \cite{IKR19} without a proof. We need to first introduce some useful notations. 

Let $q = n-d$ and $V$ be the linear subspace of $\R^n$ spanned by $A$. Now there exist distinct unit vectors $e_{1_A},\dots,e_{q_A} \in \R^n$ along the positive coordinate axes, $\{e_1,\dots,e_n\}$, such that $e_{i_A} \notin V$ and $E_A := \{v_1,\dots,v_d,e_{1_A},\dots,e_{q_A}\}$ spans $\R^n$. We define $\phi_A: [0,1]^n \to \R^n$ by the formula 
\begin{equation}\label{eq:coords}\phi_A(t_1,\dots,t_q,s_1,\dots,s_d) = t_1e_{1_A}+\cdots+t_q e_{q_A} + s_1v_1+\cdots + s_dv_d.\end{equation}
We may write $T = (t_1,\dots,t_q)$, $S = (s_1,\dots,s_d)$ and $dx = dSdT = dTdS$ to shorten notation.

\begin{remark} These coordinates are not unique, but we suppose that we have fixed some $e_{1_A},\dots,e_{q_A}$ for every $A \in G_d^n$. The specific choice is not important in our method.
\end{remark}

Next we discuss some elementary properties of the coordinates $\phi_A$. The image of $\phi_A$ is an $n$-parallelepiped when interpreted in $\R^n$. A simple calculation shows that $\abs{\det(D\phi_A)} = \abs{\det(v_1,\dots,v_n,e_{1_A},\dots,e_{q_A})} \in \Z_+$, which is also equal to the volume of the $n$-parallelepiped spanned by $E_A$. The corners of the parallelepiped, $\phi_A(T,S)$ with $T \in \{0,1\}^q, S \in \{0,1\}^d$, have integer coordinates as well. It can be argued that the coordinates (\ref{eq:coords}) wrap around the torus $\abs{\det(D\phi_A)}$ times when projected into $\T^n$, i.e.~$\abs{\det(D\phi_A)} = \abs{\phi_A^{-1}(x)}$ for any $x \in \T^n$.

Let us denote the Lebesgue measure on $\T^n$ by $dm$ and on $[0,1]^n$ by $dx$. We thus have the change of coordinates formula for integrals of measurable functions in the form of
\begin{equation}\label{eq:intsubs} \begin{split} \int_{\T^n} fdm &= \frac{1}{\abs{\det(D\phi_A)}}\int_{[0,1]^n} f \circ \phi_A \abs{\det(D\phi_A)} dx \\
&= \int_{[0,1]^n} f \circ \phi_A dx.\end{split}
\end{equation} 
The formula (\ref{eq:intsubs}), in a slightly different form, was used in the proofs given in \cite{IKR19}. The connection to \cite{IKR19} is explained with more details in remark \ref{rmk:example}.

\begin{remark}\label{rmk:example} Let $n = 2, d = 1$, $v =(v^1,v^2) \in \Z^2 \setminus \{0\}$ and $A = \{v\}$. Suppose that $v$ is not parallel to $e_1$, which in turn implies that $v^2 \neq 0$. If we choose $E_A = \{e_1\}$, then the formula $\abs{\det(D\phi_A)} = \abs{v^2}$ holds and it is easy to check that the coordinates wrap $\abs{v^2}$ times around $\T^2$. If $v$ is parallel to $e_1$, then one chooses $E_A = \{e_2\}$ instead of $e_1$. This is in-line with the formulas derived in \cite{IKR19} but there the coordinates were scaled so that they wrap around $\T^2$ exactly once.
\end{remark}

Now we are ready to prove lemma \ref{lem:L1lemma}.

\begin{proof}[Proof of lemma \ref{lem:L1lemma}] The properties \ref{item:L1prop1} and \ref{item:L1prop3} follow easily from the definitions, and the proofs are thus omitted. 

We show first that the mapping $R_{d,A}$ is well defined by the formula (\ref{eq:aeL1def}). Let $\tilde{0} = (0,\dots,0) \in \R^d$. We get from Fubini's theorem and the formula (\ref{eq:intsubs}) that
\begin{equation}\int_{\T^n} f dm = \int_{[0,1]^q} R_{d,A}f(\phi_A(T,\tilde{0})) dT\end{equation} and $R_{d,A}f(\phi_A(T,\tilde{0})) \in L^1([0,1]^q)$. It follows from the definition (\ref{eq:aeL1def}) of $R_{d,A}f$ that \begin{equation}\label{eq:shiftHyper}R_{d,A}f(\phi_A(T,\tilde{0})) = R_{d,A}f(\phi_A(T,S))\end{equation} for all $S \in \R^d$.

We show that $R_{d,A}f$ is a measurable function. Suppose for simplicity that $f$ is real valued. Let $\alpha > 0$ and define the sets
\begin{equation}X_\alpha = \{\,T \in [0,1]^q \,;\, R_{d,A}f(\phi_A(T,\tilde{0})) > \alpha\,\}.\end{equation} We have already proved that the set $X_\alpha$ is measurable for any $\alpha >0$. Now we get from the formula (\ref{eq:shiftHyper}) that \begin{equation}\{\,p \in [0,1]^n \,;\, R_{d,A}f(\phi_A(p)) > \alpha \,\} = X_\alpha \times [0,1]^d.\end{equation} The set $X_\alpha \times [0,1]^d$ is measurable as a product of measurable sets. Since $\phi_A$ is a smooth change of coordinates, we first find that $\phi_A(X_\alpha \times [0,1]^d)$ is measurable, and thus $R_{d,A}f$ is measurable. If $f$ is complex valued, then the above argument can be done separately for the real and imaginary parts as $R_{d,A}$ is linear.

Now we are ready to prove the property \ref{item:L1prop2}. Suppose that $f \in L^p(\T^n)$ and $p \in [1,\infty)$. The formulas (\ref{eq:intsubs}) and (\ref{eq:shiftHyper}), and Hölder's inequality give
\begin{equation}\begin{split}\int_{\T^n} \abs{R_{d,A}f}^p dm &= \int_{[0,1]^q}\int_{[0,1]^d} \abs{R_{d,A}f \circ \phi_A}^p dx \\ &= \int_{[0,1]^q} \abs{(R_{d,A}f)(\phi_A(T,\tilde{0}))}^pdT \\
&\leq \int_{[0,1]^q} (R_{d,A}\abs{f}^p)(\phi_A(T,\tilde{0}))dT  \\&= \norm{f}_{L^p(\T^n)}^p < \infty.\end{split}\end{equation} Hence Tonelli's theorem implies that $R_{d,A}f \in L^p(\T^n)$. If $p = \infty$, then trivially $\norm{R_{d,A}f}_{L^\infty(\T^n)} \leq \norm{f}_{L^\infty(\T^n)}$.\end{proof}

\subsubsection{Mapping properties of periodic Radon transforms}

We first recall the inversion formula in \cite{I15}. If one writes the formula \cite[Eq. (2)]{I15} in terms of the periodic subspaces, it gives the following theorem.

\begin{theorem}[Eq. (2) in \cite{I15}]\label{thm:ilmavirta} Let $f \in \mathcal{T}'$, $k \in \Z^n$ and $A \in \Gr(d,n)$. Then $\widehat{R_df}(k,A) = \hat{f}(k) \delta_{k\bot A}$, where
\begin{equation}
\delta_{k\bot A} = \begin{cases} 1 & \text{if $k \bot A$}\\
0 & \text{otherwise}.\end{cases}
\end{equation}\end{theorem} 
It is evident that for every $k \in \Z^n$ there exists $A \in \Gr(d,n)$ such that $k \bot A$, see \cite[p. 11]{A11} and \cite[Lemma 9]{I15}. This directly gives a reconstructive inversion procedure for $R_d$. In section \ref{sec:L1}, we derive new inversion formulas which might provide computational advantage in practice (cf. \cite{IKR19} when $n = 2$ and $d=1$).

\begin{lemma}\label{thm:commutation} Let $A \in \Gr(d,n)$.
\begin{enumerate}[label=(\roman*)]
\item If $P: \mathcal{T}' \to \mathcal{T}'$ acts as a Fourier multiplier $(p_k)_{k \in \Z^n}$, then $[P,R_{d,A}] = 0$.\label{item:comprop1}
\item $R_{d,A}: L_s^p(\T^n) \to L_s^p(\T^n)$ is $1$-Lipschitz for any $p \in [1,\infty]$.\label{item:comprop2}
\end{enumerate}
\end{lemma}
\begin{proof} \ref{item:comprop1} This is a simple application of theorem \ref{thm:ilmavirta}. We calculate that \begin{equation}\widehat{R_d(Pf)}(k,A) = \widehat{Pf}(k)\delta_{k\bot A} = p_k\hat{f}(k)\delta_{k\bot A} = \widehat{P(R_df)}(k,A).\end{equation}

\ref{item:comprop2} Suppose that $f \in L_s^p(\T^n)$. Now $ h := (1-\Delta)^{s/2}f \in L^p(\T^n)$. Notice that $R_{d,A}h \in L^p(\T^n)$ by lemma \ref{lem:L1lemma}. We have by the property \ref{item:comprop1} that $(1-\Delta)^{s/2}R_{d,A}f = R_{d,A}h \in L^p(\T^n).$ Hence $R_{d,A}f \in L_s^p(\T^n)$. We can conclude that
\begin{equation}\norm{R_{d,A}f}_{L_s^p(\T^n)} = \norm{R_{d,A}h}_{L^p(\T^n)} \leq \norm{h}_{L^p(\T^n)} = \norm{f}_{L_s^p(\T^n)}\end{equation}
by lemma \ref{lem:L1lemma}.
\end{proof}

The next lemma generalizes \cite[Proposition 11]{IKR19} to many different directions.

\begin{lemma}\label{cor:cont} Let $p \in [1,\infty]$. \begin{enumerate}[label=(\roman*)]
\item Let $l \in [1,\infty)$. Suppose that for any $A \in \Gr(d,n)$ there exists $C_A > 0$ such that $w(k,A) = C_A$ for every $k \bot A$. Moreover, suppose that \begin{equation}C_w^l := \sum_{A\in \Gr(d,n)} C_A^l < \infty.\end{equation} Then the Radon transform $R_{d}: L_s^p(\T^n) \to L_s^{p,l}(X_{d,n};w)$ is $C_w$-Lipschitz.\label{item:contprop1}
\item Suppose that for any $A \in \Gr(d,n)$ there exists $C_A > 0$ such that $w(k,A) = C_A$ for every $k \bot A$. Moreover, suppose that \begin{equation}C_w = \sup_{A \in \Gr(d,n)} C_A < \infty.\end{equation} Then the Radon transform $R_{d}: L_s^p(\T^n) \to L_s^{p,\infty}(X_{d,n};w)$ is $C_w$-Lipschitz.\label{item:contprop2}
\item Suppose that there exists $C_w > 0$ such that \begin{equation}\sum_{A \in \Omega_k} w(k,A)^2 \leq C_w^2, \quad \Omega_k := \{\, A \in \Gr(d,n) \,;\, k\bot A\,\}\end{equation}
for any $k \in \Z^n$. Then the Radon transform $R_{d}: H^s(\T^n) \to L_s^{2,2}(X_{d,n};w)$ is $C_w$-Lipschitz.\label{hilbcase}
\end{enumerate}
\end{lemma}
\begin{proof} \ref{item:contprop1} We have that
\begin{equation}\label{eq:boundedness}\norm{R_{d,A}f}_{L_s^p(\T^n)} \leq \norm{f}_{L_s^p(\T^n)}\end{equation} for any $A \in \Gr(d,n)$ by lemma \ref{thm:commutation}. Theorem \ref{thm:ilmavirta} implies that
\begin{equation}\label{eq:jokueq}F_{w(\cdot,A)}R_{d,A}f(x) = \sum_{k \bot A} w(k,A)\hat{f}(k)e^{2\pi i k\cdot x}.\end{equation} This gives that $F_{w(\cdot,A)}R_{d,A}f = C_AR_{d,A}f$. Now it follows from \eqref{eq:boundedness} and the definition of $C_w^l$ that
\begin{equation}\begin{split}\norm{R_df}_{L_s^{p,l}(X_{d,n};w)}^l &= \sum_{A \in \Gr(d,n)}C_A^l \norm{R_{d,A}f}_{L_s^p(\T^n)}^l \\
&\leq C_w^l\norm{f}_{L_s^p(\T^n)}^l.\end{split}\end{equation}

\ref{item:contprop2} A calculation similar to the proof of \ref{item:contprop1} shows that
\begin{equation}\norm{R_df}_{L_s^{p,\infty}(X_{d,n};w)} \leq \norm{f}_{L_s^p(\T^n)} \sup_{A \in \Gr(d,n)} C_A.\end{equation}

\ref{hilbcase} We have that 
\begin{equation}\begin{split}\norm{R_df}_{L_s^{2,2}(X_{d,n};w)}^2 &= \sum_{A \in \Gr(d,n)}\norm{\sum_{k\bot A} w(k,A) \vev{k}^s \hat{f}(k)e^{2\pi i k \cdot x}}_{L^2(\T^n)}^2 \\
&= \sum_{A \in \Gr(d,n)} \sum_{k\bot A} w(k,A)^2 \abs{\vev{k}^s\hat{f}(k)}^2\\
&= \sum_{k \in \Z^n} \sum_{A \in \Omega_k} w(k,A)^2 \vev{k}^{2s} \abs{\hat{f}(k)}^2 \\
&\leq C_w^2 \norm{f}_{L_s^2(\T^n)}^2 \end{split}\end{equation}
where the order of summation can be interchanged by non-negativity of the terms.\end{proof}

\begin{remark} If $d = n-1$, then the only restriction on $w$ in the case of \ref{hilbcase} is $\sum_{A \in \Gr(n-1,n)} w(0,A)^2 < \infty$. This follows since each $A \in \Gr(n-1,n)$ has a unique normal direction.
\end{remark}

\subsubsection{Adjoint and normal operators}
\label{sec:adjnorm}

Next, we study the adjoint and normal operators of $R_d$ when the image side is equipped with the Hilbert space $L_s^{2,2}(X_{d,n};w)$ satisfying the assumptions \ref{hilbcase} of lemma \ref{cor:cont}. This generalizes the considerations in \cite[Section 2.4]{IKR19} into higher dimensions and for any $1 \leq d \leq n-1$.

\begin{proof}[Proof of theorem \ref{thm:adjoint}] Let $f \in H^s(\T^n)$ and $g \in L_s^{2,2}(X_{d,n};w)$. Using the definition of the inner product (\ref{eq:innerprod}), we get
\begin{equation}\begin{split}(R_df,g)_{L_s^{2,2}(X_{d,n};w)} &= \sum_{A \in \Gr(d,n)} (F_{w(\cdot,A)}R_d f,F_{w(\cdot,A)}g)_{H^s(\T^n)}\\
&= \sum_{A \in \Gr(d,n)} \sum_{k \bot A} w(k,A)^2\vev{k}^{2s} \hat{f}(k)\hat{g}(k,A)^* \\
&= \sum_{k \in \Z^n} \sum_{A \in \Omega_k} w(k,A)^2\vev{k}^{2s} \hat{f}(k)\hat{g}(k,A)^*\\
&= \sum_{k \in \Z^n} \vev{k}^{2s} \hat{f}(k) \left(\sum_{A \in \Omega_k} w(k,A)^2\hat{g}(k,A)\right)^*\\
&=: (f,R_d^*g)_{H^s(\T^n)}\end{split}\end{equation}
where we can interchange the order of the summation by the Cauchy–Schwarz inequality as it implies that the series is absolutely convergent.

We have that \begin{equation}\begin{split}
\widehat{R_d^*R_df}(k) &= \sum_{A \in \Omega_k} w(k,A)^2 \widehat{R_df}(k,A) \\
&= \sum_{A \in \Omega_k} w(k,A)^2 \hat{f}(k)\delta_{k\bot A}\\
&= \hat{f}(k)\sum_{A \in \Omega_k} w(k,A)^2\end{split}\end{equation}
by the formula for the adjoint and theorem \ref{thm:ilmavirta}.
\end{proof}

We prove corollary \ref{cor:h2results} on inversion formulas and stability estimates next.

\begin{proof}[Proof of corollary \ref{cor:h2results}] \ref{item:h1prop1} We first calculate that
\begin{equation}\label{eq:lipeka} \norm{F_{W_k^{-1}} R_d^* g}_{H^s(\T^n)}^2 = \sum_{k \in \Z^n} \vev{k}^{2s}\frac{1}{W_k^2}\abs{\sum_{A \in \Omega_k} w(k,A)^2 \hat{g}(k,A)}^2
\end{equation} for any $g \in L_s^{2,2}(X_{d,n};w)$. The triangle inequality and Hölder's inequality for the sequences $w(k,A)$ and $w(k,A)\abs{\hat{g}(k,A)}$ over $A \in \Omega_k$ gives that
\begin{equation}\label{eq:liptoka}
\abs{\sum_{A \in \Omega_k} w(k,A)^2 \hat{g}(k,A)}^2 
\leq W_k \left(\sum_{A \in \Omega_k} w(k,A)^2 \abs{\hat{g}(k,A)}^2\right).
\end{equation}
Recall that
\begin{equation}\label{eq:lipkolmas} \norm{g}_{L_s^{2,2}(X_{d,n};w)}^2 = \sum_{k \in \Z^n} \vev{k}^{2s} \sum_{A \in \Gr(d,n)} w(k,A)^2 \abs{\hat{g}(k,A)}^2
\end{equation}
after a rearrangement of the series. We can conclude from the formulas (\ref{eq:lipeka}), (\ref{eq:liptoka}) and (\ref{eq:lipkolmas}) that $\norm{F_{W_k^{-1}}R_d^*g}_{H^s(\T^n)} \leq \frac{1}{c_w}\norm{g}_{L_s^{2,2}(X_{d,n};w)}$.

\ref{item:h1prop2} This is a simple calculation using the formula for the normal operator:
\begin{equation}(R_df,R_df)_{L_s^{2,2}(X_{d,n};w)} 
=(f,F_{W_k}f)_{H^s(\T^n)} \geq \inf_{k \in \Z^n}{W_k} \norm{f}_{H^s(\T^n)}^2\end{equation} if $f \in H^s(\T^n)$.

\ref{item:Lspstab} We have by remark \ref{rm:joku2} that $\tilde{w}$ is a weight that satisfies the assumptions of theorem \ref{thm:adjoint} and $\tilde{W}_k = 1$ for any $k \in \Z^n$. Therefore, the corresponding adjoint $R_d^{*,\tilde{w}}$ is well-defined, and $R_d^{*,\tilde{w}}R_df = f$ for any $f \in \mathcal{T}'$ by theorem \ref{thm:adjoint}.
\end{proof}

\section{Inversion formulas}
\label{sec:L1}

We have already proved one new inversion formula in corollary \ref{cor:h2results} for $H^s(\T^n)$ functions. In this section, we prove three other inversion formulas. One of the formulas generalizes the inversion formula for $R_1$ on $L^1(\T^2)$ proved in \cite[Theorem 1 and Theorem 8]{IKR19}. The second inversion formula is a corollary of the first one and remains valid for any distribution. The third inversion formula takes a slightly different approach and shows that a distribution $f \in \mathcal{T}'$ is a weighted sum of the data $R_{d,A}f$ over the set $\Gr(d,n)$. These formulas might have practical value.

\begin{proposition}[The first inversion formula]\label{prop:L1rec} Let $A\in \Gr(d,n)$ and $k \in \Z^n$. Suppose that $f \in \mathcal{T}'$ and $R_{d,A}f \in L^1(\T^2)$. If $k\bot A$, then
\begin{equation}\label{eq:invformula}\hat{f}(k) = \int_{[0,1]^q}  R_{d,A}f(\phi_A(T,0))\exp(-2\pi i (k_{1_A}t_{1_A}+\cdots+k_{q_A}t_{q_A}))dT.
\end{equation}
\end{proposition}
\begin{proof} Fubini's theorem, theorem \ref{thm:ilmavirta} and the formula (\ref{eq:intsubs}) implies that
\begin{equation}\label{eq:invforl1X}\begin{split}&\widehat{R_{d,A}f}(k) \\
&\,\,= \int_{[0,1]^q}\int_{[0,1]^d} R_{d,A}f(\phi_A(T,S))\exp(-2\pi i k \cdot \phi_A(T,S))dSdT.\end{split}\end{equation} Since $k\bot A$, a simple calculation shows that 
\begin{equation}\label{eq:invforl1A} k \cdot \phi_A(T,S) = k_{1_A}t_{1_A}+\cdots+k_{q_A}t_{q_A},\end{equation} and lemma \ref{lem:L1lemma} implies that \begin{equation}\label{eq:invforl1B} R_{d,A}f(\phi_A(T,S)) = R_{d,A}f(\phi_A(T,0))\end{equation} for a.e. $T \in [0,1]^q$.

Hence, using the formulas (\ref{eq:invforl1A}) and (\ref{eq:invforl1B}), we may simplify the formula (\ref{eq:invforl1X}) into the form
\begin{equation}\begin{split}&\widehat{R_{d,A}f}(k) \\
&\,\,= \int_{[0,1]^q}  R_{d,A}f(\phi_A(T,0))\exp(-2\pi i (k_{1_A}t_{1_A}+\cdots+k_{q_A}t_{q_A}))dT.\end{split}\end{equation} \end{proof}
\begin{remark} The proof shows that instead of choosing $S =0$, we may choose any other values for the $S$-coordinates as well.
\end{remark}

We immediately get the following corollary from proposition \ref{prop:L1rec} and lemma \ref{lem:L1lemma}.
\begin{corollary}\label{cor:invform} Suppose that $f \in L^1(\T^n)$. Then the inversion formula (\ref{eq:invformula}) is valid.
\end{corollary}

\begin{remark} One could prove corollary \ref{cor:invform} directly without using lemma \ref{lem:L1lemma} and theorem \ref{thm:ilmavirta} (or proposition \ref{prop:L1rec}). This proof is given for the geodesic X-ray transform in \cite{IKR19} and it could be adapted to this setting as well.
\end{remark}

Recall that the structure theorem of periodic distributions \cite[Theorem 2.4.5]{S13} states that for any $f \in \mathcal{T}'$ there exist $h \in C(\T^n)$ and $s \geq 0$ such that
\begin{equation}\label{eq:struct}
f=(1-\Delta)^sh.\end{equation} We get another corollary of proposition \ref{prop:L1rec} and lemma \ref{thm:commutation}.
\begin{corollary}[The second inversion formula]\label{cor:dist} Let $A\in \Gr(d,n)$ and $k \in \Z^n$. Suppose that $f \in \mathcal{T}'$ and $f=(1-\Delta)^sh$, $h \in C(\T^n)$. If $k \bot A$, then
\begin{equation} \hat{f}(k) = \vev{k}^{2s}\widehat{R_{d,A}h}(k) = \widehat{R_{d,A}f}(k)\end{equation} where $\widehat{R_{d,A}h}(k)$ can be calculated by the formula (\ref{eq:invformula}).
\end{corollary}

We now prove our third inversion formula stated in the introduction.
\begin{proof}[Proof of theorem \ref{thm:recwithoutFourier}]  Using theorem \ref{thm:ilmavirta}, we calculate that
\begin{equation}\mathcal{F}(F_{w(\cdot,A)}R_{d,A}f)(k) = w(k,A)\hat{f}(k)\delta_{k\bot A}.\end{equation}
Hence, we get
\begin{equation}\begin{split}\mathcal{F}\left(\sum_{A \in \Gr(d,n)} F_{w(\cdot,A)}R_{d,A}f\right)(k) &= \sum_{A \in \Gr(d,n)} w(k,A)\hat{f}(k)\delta_{k\bot A} \\
&=\hat{f}(k) \sum_{A \in \Omega_k} w(k,A) \\
&=\hat{f}(k)\\
\end{split}\end{equation}

Suppose now that $d = n-1$ and $\hat{f}(0) = 0$. Notice that $\abs{\Omega_k} = 1$ if $k\neq0$ and $\Omega_0 = \Gr(n-1,n)$. Hence, the formula (\ref{eq:recwithoutFourier0mean}) follows by choosing any weight $w$ such that \begin{equation}\sum_{A \in \Gr(n-1,d)} w(0,A) = 1, w(0,A) \geq 0,\end{equation} and $w(k,A) = 1$ for any $A \in \Gr(n-1,n)$ and $k\neq0$.
\end{proof}

\section{Stability estimates and regularization methods}\label{sec:stabtikh}

In this section, we look at stability estimates for functions in the Bessel potential spaces when $p \neq \infty$. We also generalize the Tikhonov regularization methods developed in \cite{IKR19}. In the Tikhonov regularization part, we restrict our study to the functions in $H^s(\T^n)$, as done in \cite{IKR19}. Our results on regularization are new for any $1 \leq d \leq n-1$ when $n \geq 3$, and the stability estimates are new in any dimension.

\subsection{Stability estimates and the Sobolev inequality}

Recall that in corollary \ref{cor:h2results} we obtained the estimate \begin{equation}\norm{f}_{H^s(\T^n)}^2 \leq \frac{1}{c_w^2}\norm{R_df}_{L_s^{2,2}(X_{d,n};w)}^2\end{equation} if the weight $w$ is such that the normal operator $R_d^*R_d$ has a uniform lower bound $\frac{1}{c_w^2}$ as a Fourier multiplier. The condition on the weight $w$ is that $c_w^2 \leq W_k = \sum_{A \in\Omega_k} w(k,A)^2 \leq C_w^2$ for some uniform $c_w, C_w > 0$. This implies stability on $L_s^p(\T^n)$ if $p \leq2$, as we will show later. We can reach stability estimates for $p > 2$ using the Sobolev inequality on $\T^n$. 

\begin{theorem}[Sobolev inequality \cite{S83}] Let $f \in \mathcal{T}'$. Suppose that $s > 0$ and $1 < q < p < \infty$ satisfy $s/n \geq q^{-1} - p^{-1}$. Then \begin{equation}\label{eq:sobineq}\norm{f}_{L^p(\T^n)} \leq C \norm{f}_{L_s^q(\T^n)}\end{equation} for some $C > 0$ that does not depend on $f$.
\end{theorem}
A proof of the Sobolev inequality on $\T^n$ is given in \cite[Corollary 1.2]{BT13}.

\begin{lemma}\label{lem:sobey} Let $l \in [1,\infty]$ and $g: \Gr(d,n) \to \mathcal{T}'$.
\begin{enumerate}[label=(\roman*)]
\item \label{item:ekaEY}  If $t \in \R$, $s > 0$, and $1 < q < p < \infty$ satisfy $s/n \geq q^{-1} -p^{-1}$, then
\begin{equation}
\norm{g}_{L_{t}^{p,l}(X_{d,n};w)} \leq C\norm{g}_{L_{t+s}^{q,l}(X_{d,n};w)}\label{eq:sobinequalityData}
\end{equation}
for some $C > 0$ that does not depend on $g$.

\item \label{item:tokaEY} If $1\leq p < q \leq \infty$, then for any $s \in \R$ holds
\begin{equation}
\norm{g}_{L_s^{p,l}(X_{d,n};w)} \leq \norm{g}_{L_s^{q,l}(X_{d,n};w)}. \label{eq:perusineq}
\end{equation}
\end{enumerate}
\end{lemma}
\begin{proof} \ref{item:ekaEY} We have
\begin{equation} \norm{g(\cdot,A)}_{L^p(\T^n;w(\cdot,A))} \leq C\norm{g(\cdot,A)}_{L_s^q(\T^n;w(\cdot,A))}\label{eq:jokuosa}
\end{equation}
for any $A \in \Gr(d,n)$ by the Sobolev inequality where $C > 0$ does not depend on $f$, $A$ and $w$. Now (\ref{eq:sobinequalityData}) with $t=0$ follows from the definition of the norms $\norm{\cdot}_{L_s^{q,l}(X_{d,n};w)}$ and the inequality (\ref{eq:jokuosa}).

Fix any $z \in \R$. Define then the function $\tilde{g}: \Gr(d,n) \to \mathcal{T}'$ by the formula $\tilde{g}(\cdot,A) = (1-\Delta)^{z/2}g(\cdot,A)$. Now (\ref{eq:sobinequalityData}) with $t=0$ implies
\begin{equation}
\norm{g}_{L_{z}^{p,l}(X_{d,n};w)} = \norm{\tilde{g}}_{L_{0}^{p,l}(X_{d,n};w)} \leq C\norm{\tilde{g}}_{L_{s}^{q,l}(X_{d,n};w)} = C\norm{g}_{L_{z+s}^{q,l}(X_{d,n};w)}.
\end{equation}

\ref{item:tokaEY} The inequality (\ref{eq:perusineq}) can be proved similarly. Now the Sobolev inequality is replaced by the inequality $\norm{f}_{L_s^p(\T^n)} \leq \norm{f}_{L_s^q(\T^n)}$, which holds since $m(\T^n) = 1$ and $p \leq q$.
\end{proof}

Theorem \ref{thm:adjoint} and lemma \ref{lem:sobey} imply the following, slightly more general, shifted stability estimates.

\begin{proposition}[Shifted stability estimates]\label{prop:stabilitymain} Let $w$ be a weight such that $c_w^2 \leq W_k \leq C_w^2$ for some uniform constants $c_w, C_w > 0$. Let $f \in \mathcal{T}'$, $s \in \R$, and $s(p,n) := n\abs{\frac{p-2}{2p}}$. \begin{enumerate}[label=(\roman*)]

\item\label{stabL2} If $1 < p \leq 2$, then \begin{equation}\label{eq:stabthm} \norm{f}_{L_s^p(\T^n)} \leq C_1\norm{R_df}_{L_s^{2,2}(X_{d,n};w)} \leq C_2\norm{R_df}_{L_{s+s(p,n)}^{p,2}(X_{d,n};w)},\end{equation} where $C_1,C_2 > 0$ do not depend on $f$. If $p = 1$, then the first inequality of (\ref{eq:stabthm}) holds.

\item\label{stabShifted} If $2 \leq p < \infty$, then 
\begin{equation}\norm{f}_{L_s^p(\T^n)} \leq C_1\norm{R_df}_{L_{s+s(p,n)}^{2,2}(X_{d,n};w)} \leq C_2\norm{R_df}_{L_{s+s(p,n)}^{p,2}(X_{d,n};w)},\end{equation}
where $C_1,C_2 > 0$ do not depend on $f$.
\end{enumerate}
\end{proposition}
\begin{proof} \ref{stabL2} Suppose that $f \in \mathcal{T}'$ and $1 \leq p \leq 2$. Let $h = (1-\Delta)^{s/2}f$. We have that $\norm{h}_{L^p(\T^n)} \leq \norm{h}_{L^2(\T^n)}$ since $p \leq 2$ and $m(\T^n) = 1$. This implies that $\norm{f}_{L_s^p(\T^n)} \leq \norm{f}_{L_s^2(\T^n)}$. Now the first inequality follows from corollary \ref{cor:h2results}. 

Suppose additionally that $1 < p < 2$. Choose $s^p = n\frac{2-p}{2p} >0$ in the part \ref{item:ekaEY} of lemma \ref{lem:sobey}. Now it holds that 
\begin{equation}\norm{R_df}_{L_s^{2,2}(X_{d,n};w)} \leq  \norm{R_df}_{L_{s+s^p}^{p,2}(X_{d,n};w)}
\end{equation}
for any $s \in \R$.

\ref{stabShifted} Suppose that $f \in \mathcal{T}'$ and $p > 2$. Choose in the Sobolev inequality (\ref{eq:sobineq}) that $q = 2$. Now we can calculate that the Sobolev inequality is valid if $s \geq n\frac{p-2}{2p}$. Let us define that $s_p = n\frac{p-2}{2p} > 0$. Hence, $\norm{f}_{L^p(\T^n)} \leq C \norm{f}_{H^{s_p}(\T^n)}$.

Let now $s \in \R$ and $f \in L_s^p(\T^n)$. We then have that
\begin{equation}\begin{split}\norm{f}_{L_s^p(\T^n)}&=\norm{(1-\Delta)^{s/2}f}_{L^p(\T^n)} \\
&\leq C\norm{(1-\Delta)^{s/2}f}_{H^{s_p}(\T^n)} = C\norm{f}_{H^{s+s_p}(\T^n)}.\end{split}\end{equation} Now the first inequality follows from the part \ref{stabL2} of the theorem. The second inequality follows from the part \ref{item:tokaEY} of lemma \ref{lem:sobey} since $p > 2$. 
\end{proof}

\begin{remark} For any $f \in \mathcal{T}'$ there exists $s \geq 0$ such that $f \in L_{-s}^p(\T^n)$ for any $p \in [1,\infty]$ by the structure theorem of periodic distributions.
\end{remark}

\subsection{Tikhonov minimization problem}

We will show that $P_{w,s-r}^\alpha R_d^*g$ is the unique minimizer of (\ref{eq:tikh}) when $l = 2$. We first analyze the regularity properties of $P_{w,z}^\alpha$ and $P_{w,s-r}^\alpha R_d^*$. Then we understand which space the regularized reconstruction $P_{w,s-r}^\alpha R_d^*g$ lives in when $g \in L_r^{2,2}(X_{d,n};w)$. First of all, $R_d^*: L_r^{2,2}(X_{d,n};w) \to H^r(\T^n)$. On the other hand, $P_{w,z}^\alpha: H^r(\T^n) \to H^{r+2z}(\T^n)$ for any $r, z \in \R$ since $W_k$ is uniformly bounded from below. We conclude that $P_{w,s-r}^\alpha R_d^*: L_r^{2,2}(X_{d,n};w) \to H^{2s-r}(\T^n)$.

We are not ready to prove theorem \ref{thm:tikhmin}. The proof uses the same ideas as the proof of \cite[Theorem 2]{IKR19}. The proof presented here also explains some missing details about the splitting of the minimization problem into the real and imaginary parts in (\ref{eq:tikheqZ}), (\ref{eq:realpart}) and (\ref{eq:imagpart}). This is one of the crucial parts of the proof of \cite[Theorem 2]{IKR19} though it is not mentioned at all in \cite{IKR19}.

\begin{proof}[Proof of theorem \ref{thm:tikhmin}]
We have that
\begin{equation}\label{eq:tikheqB}\begin{split}&\norm{R_df -g}_{L_r^{2,2}(X_{d,n};w)}^2 \\
&\,\,= \sum_{A \in \Gr(d,n)} \sum_{k \bot A} \vev{k}^{2r}w(k,A)^2\abs{\hat{f}(k) - \hat{g}(k,A)}^2 \\
&\,\,+ \sum_{A \in \Gr(d,n)}\sum_{k \not\bot A} \vev{k}^{2r}w(k,A)^2\abs{\hat{g}(k,A)}^2.\end{split}\end{equation} Since the second term of (\ref{eq:tikheqB}) is independent of $f$, it can be neglected in the minimization problem (\ref{eq:tikh}). On the other hand,
\begin{equation}\begin{split}&\sum_{A \in \Gr(d,n)} \sum_{k \bot A} \vev{k}^{2r}w(k,A)^2 \abs{\hat{f}(k) - \hat{g}(k,A)}^2 \\
&\,\,= \sum_{k \in \Z^n} \vev{k}^{2r}\sum_{A \in \Omega_k} w(k,A)^2\abs{ \hat{f}(k) - \hat{g}(k,A)}^2
.\end{split}\end{equation}

We next expand the term \begin{equation}\alpha \norm{f}_{H^s(\T^n)}^2
= \alpha \sum_{k \in \Z^n} \vev{k}^{2s} \abs{\hat{f}(k)}^2.\end{equation} We can conclude that a solution to the minimization problem (\ref{eq:tikh}) is a minimizer of
\begin{equation}\label{eq:tikheqA}\sum_{k \in \Z^n} \vev{k}^{2r} \left(\alpha\vev{k}^{2s-2r}\abs{\hat{f}(k)}^2 + \sum_{A \in \Omega_k} w(k,A)^2\abs{\hat{f}(k)-\hat{g}(k,A)}^2\right).\end{equation} Hence, a minimizer of (\ref{eq:tikheqA}) must minimize \begin{equation}H_k(f) :=\label{eq:tikheqZ}\alpha\vev{k}^{2s-2r}\abs{\hat{f}(k)}^2 + \sum_{A \in \Omega_k} w(k,A)^2\abs{\hat{f}(k)-\hat{g}(k,A)}^2\end{equation} for each $k \in \Z^n$.

To proceed, we need to minimize the real part and the imaginary part of (\ref{eq:tikheqZ}) separately. Let us write the real and imaginary parts of the involved terms simply as $f_r(k) := \Re(\hat{f}(k))$, $f_i(k) := \Im(\hat{f}(k))$, $g_r(k,A) := \Re(\hat{g}(k,A))$ and $g_i(k,A) := \Im(\hat{g}(k,A))$ to keep our notation shorter. Now, we define the operators
\begin{equation}\label{eq:realpart}R_k(f) := \alpha \vev{k}^{2s-2r} f_r(k)^2 + \sum_{A \in \Omega_k} w(k,A)^2(f_r(k)-g_r(k,A))^2\end{equation} and \begin{equation}\label{eq:imagpart}I_k(f) := \alpha \vev{k}^{2s-2r} f_i(k)^2 + \sum_{A \in \Omega_k} w(k,A)^2(f_i(k)-g_i(k,A))^2.\end{equation} These functions have the property that $R_k(f) + I_k(f) = H_k(f)$. Moreover, if $H_k$ is minimized, then $R_k$ and $I_k$ are minimized, and vice versa.

We show how the minimization is done for the real part. As the minimization for the imaginary part is similar, we do not repeat the calculations twice. We expand the second term of (\ref{eq:realpart}), and get
\begin{equation}\label{eq:tikheqY}\begin{split} &\sum_{A \in \Omega_k} w(k,A)^2(f_r(k)-g_r(k,A))^2 \\
&\,\,= W_k f_r(k)^2 -2f_r(k)\sum_{A \in \Omega_k} w(k,A)^2 g_r(k,A) + \sum_{A \in \Omega_k}w(k,A)^2 g_r(k,A)^2.\end{split}\end{equation} The last term of (\ref{eq:tikheqY}) does not depend on $f$, so it can be neglected in the minimization. Thus, we have arrived to the minimization problem 
\begin{equation}\label{eq:tikheqX}-2f_r(k)\sum_{A \in \Omega_k} w(k,A)^2 g_r(k,A) + (W_k+\alpha\vev{k}^{2s-2r})f_r(k)^2.\end{equation}

Simple calculus shows that the minimizer of (\ref{eq:tikheqX}) is
\begin{equation}f_r(k) = \frac{\sum_{A \in \Omega_k} w(k,A)^2 g_r(k,A)}{W_k+\alpha\vev{k}^{2s-2r}} = \Re(\mathcal{F}(P_{w,s-r}^\alpha R_d^*g)(k)).\end{equation} We can similarly calculate that the unique minimizer of the minimization problem associated to the imaginary part (\ref{eq:imagpart}) is $f_i(k) = \Im(\mathcal{F}(P_{w,s-r}^\alpha R_d^*g)(k))$. This shows that the unique minimizer of (\ref{eq:tikheqZ}) satisfies $\hat{f}(k) = \mathcal{F}(P_{w,s-r}^\alpha R_d^*g)(k)$. 

Hence, the unique minimizer of (\ref{eq:tikh}) is $f = P_{w,s-r}^\alpha R_d^*g$. The claimed regularity of $f$ follows from the discussion preceding the proof.
\end{proof}

\begin{remark} If $l \neq 2$, the analysis of the Tikhonov minimization problem becomes more difficult but it might still be possible to adapt the method also in that case (when $p = 2$).
\end{remark}

\subsection{Regularization strategies}

Let~$X$ and~$Y$ be subsets of Banach spaces and $F: X \to Y$ a continuous mapping. A family of continuous maps $\mathcal{R}_\alpha: Y \to X$ with $\alpha \in (0,\alpha_0]$, $\alpha_0 > 0$, is called a \textit{regularization strategy} if $\lim_{\alpha \to 0} \mathcal{R}_\alpha (F(x)) = x$ for any $x \in X$. A choice of regularization parameter $\alpha(\epsilon)$ with $\lim_{\epsilon\to0}\alpha(\epsilon) = 0$ is called \textit{admissible} if 
\begin{equation}
\lim_{\epsilon \to 0}
\sup_{y \in Y} \left\{ \norm{\mathcal{R}_{\alpha(\epsilon)}y -x}_X \,;\, \norm{y-F(x)}_Y \leq \epsilon\,\right\}
= 0
\end{equation}
holds for any $x \in X$ \cite{EHN96,K11}. 

We will show that the solution found in theorem \ref{thm:tikhmin} to the Tikhonov minimization problem (\ref{eq:tikh}) is an admissible regularization strategy with a quantitative stability estimate. Our proof follows that of \cite[Theorem 3]{IKR19}.

\begin{proof}[Proof of theorem \ref{thm:regstrat}] Let $\alpha > 0$. Theorem \ref{thm:adjoint} implies that
\begin{equation}\label{eq:regeq1}P_{w,s}^{\alpha} R_d^*(R_df+g)-f = (P_{w,s}^{\alpha}F_{W_k} -\text{Id})f + P_{w,s}^{\alpha} R_d^*g.\end{equation} To estimate the first term on the right hand side of (\ref{eq:regeq1}), we calculate that 
\begin{equation}\label{eq:fourmult}P_{w,s}^{\alpha}F_{W_k} -\text{Id} = -\frac{\alpha W_k^{-1}\vev{k}^{2s}}{1+\alpha W_k^{-1}\vev{k}^{2s}}\end{equation} 
as a Fourier multiplier. This shows that $\norm{P_{w,s}^{\alpha}F_{W_k} -\text{Id}}_{H^r(\T^n)\to H^r(\T^n)} = 1$ as $W_k$ is bounded from below and above. It follows from the dominated convergence theorem that $\norm{(P_{w,s}^{\alpha}F_{W_k} -\text{Id})f}_r^2 \to 0$ as $\alpha \to 0$ if $f \in H^r(\T^n)$.

Suppose that $\norm{g}_{L_t^{2,2}(X_{d,n};w)} \leq \epsilon$. We have that $\norm{R_d^*} = \norm{R_d} = C_w$ by lemma \ref{cor:cont}. Hence $\norm{R_d^*g}_{H^t(\T^n)}^2 \leq C_w^2\epsilon^2$. This implies that
\begin{equation}\begin{split}\norm{P_{w,s}^{\alpha} R_d^*g}_{H^r(\T^n)}^2 &\leq C_w^2\epsilon^2\sup_{k \in \Z^n} \left(\frac{W_k^{-1}}{1+\alpha W_k^{-1}\vev{k}^{2s}}\right)^{2}\vev{k}^{2r-2t} \\
&\leq C_w^2\epsilon^2c_w^{-4}\sup_{k \in \Z^n} \left(\frac{1}{1+\alpha C_w^{-2}\vev{k}^{2s}}\right)^{2}\vev{k}^{2r-2t} \\
&\leq C_w^6c_w^{-4}\alpha^{-2}\epsilon^2\end{split}\end{equation} where the last inequality follows using $-4s+2r-2t \leq 0$. We can conclude that
\begin{equation}\label{eq:scndterm}\norm{P_{w,s}^{\alpha} R_d^*g}_{H^r(\T^n)} \leq C_w^3c_w^{-2}\frac{\epsilon}{\alpha}.\end{equation} This shows that choosing $\alpha = \sqrt{\epsilon}$ gives a regularization strategy.

Suppose now that $\delta > 0$. The proof of the estimate (\ref{eq:quantitative}) is similar to that of \cite{IKR19}. Using the formula (\ref{eq:fourmult}), we get that
\begin{equation}\norm{P_{w,s}^{\alpha}F_{W_k} -\text{Id}}_{H^{r+\delta}(\T^n)\to H^r(\T^n)} = \sup_{k \in \Z^n} \frac{\alpha W_k^{-1}\vev{k}^{2s-\delta}}{1+\alpha W_k^{-1}\vev{k}^{2s}}.\end{equation} We can estimate the norm by defining the functions
\begin{equation}F_k(x) := \frac{\alpha W_k^{-1} x^{2s-\delta}}{1+\alpha W_k^{-1}x^{2s}}.\end{equation} The formula \cite[Eq. (38)]{IKR19} implies that the maximum value of $F_k$ is $(W_k^{-1}\alpha)^{\delta/2s}C(\delta/2s)$ if $\alpha \leq W_k(2s/\delta -1)$. We see that $\alpha \leq W_k(2s/\delta -1)$ holds as we assumed that $\alpha \leq c_w^2(2s/\delta -1)$.

We obtain that
\begin{equation}\begin{split}&\norm{(P_{w,s}^\alpha F_{W_k} -\text{Id})}_{H^{r+\delta}(\T^n)\to H^r(\T^n)} \\
&\,\,\leq \sup_{k \in \Z^n, x \in \R} F_k(x) \leq (c_w^{-2}\alpha)^{\delta/2s}C(\delta/2s).\end{split}\end{equation}
Hence
\begin{equation}\label{eq:finall}\norm{(P_{w,s}^\alpha F_{W_k} -\text{Id})f}_{H^r(\T^n)} \leq (c_w^{-2}\alpha)^{\delta/2s}C(\delta/2s)\norm{f}_{H^{r+\delta}(\T^n)}.\end{equation} Now the formulas (\ref{eq:scndterm}) and (\ref{eq:finall}) imply the quantitative estimate (\ref{eq:quantitative}).
\end{proof}
\bibliographystyle{abbrv}
\bibliography{19149FIN}
\end{document}